\documentclass{article}
\usepackage{amsfonts}
\usepackage{amsmath}
\usepackage{amssymb}
\usepackage{graphicx}

\newtheorem{theorem}{Theorem}

\newtheorem{corollary}[theorem]{Corollary}

\newtheorem{definition}[theorem]{Definition}

\newtheorem{remark}[theorem]{Remark}

\newenvironment{proof}[1][Proof]{\noindent\textbf{#1.} }{\ \rule{0.5em}{0.5em}}
\input{tcilatex}

\begin{document}

\title{Majorization under constraints and bounds of the second Zagreb index}
\author{\emph{Monica Bianchi $^{\alpha}$} \thanks{%
e-mail: monica.bianchi@unicatt.it} \and \emph{Alessandra Cornaro $^{\alpha}$}
\thanks{%
Corresponding author, e-mail: alessandra.cornaro@unicatt.it,
Tel:+390272343683, Fax:+390272342324} \and \emph{Anna Torriero$^{\alpha}$} 
\thanks{%
e-mail: anna.torriero@unicatt.it} }
\date{ }
\maketitle

\begin{abstract}
In this paper we present a theoretical analysis in order to establish
maximal and minimal vectors with respect to the majorization order of
particular subsets of $\Re ^{n}.$ Afterwards we apply these issues to the
calculation of bounds for a topological descriptor of a graph known as the
second Zagreb index. Finally, we show how our bounds may improve the results
obtained in the literature, providing some theoretical and numerical
examples.


\end{abstract}

\centerline{$^{\alpha}$ Department of Mathematics and Econometrics, Catholic University, Milan,
Italy.}

\section{Introduction}

The notion of majorization ordering was introduced by Hardy, Littlewood and
Polya (\cite{Hardy}) and is closely connected with the economic theory of
disparity indices (\cite{Atkinson}). But this concept can first be found in
Schur (\cite{Schur}) who investigated functions which preserve the
majorization order, the so-called Schur-convex functions. Using this
property and characterizing maximal and minimal vectors with respect to
majorization order under suitable constraints, many inequalities involving
such functions can be derived (\cite{Marshall}). A significant application
of this approach concerns the localization of ordered sequences of real
numbers as they occur in the problem of finding estimates of eigenvalues of
a matrix (\cite{Bianchi}, \cite{Pan}, \cite{Tarazaga} and \cite{Tarazaga2}).
Another field of interest concerns the network analysis, where the same
methodology can be useful applied in order to provide bounds for some
topological indicators of graphs which can be usefully expressed as a
Schur-convex function, in terms of the degree sequence of the graph (see 
\cite{Grassi}).

\noindent In this paper, after some preliminary definitions and notations,
we perform a theoretical analysis aimed at determining maximal and minimal
vectors with respect to the majorization order of suitable subsets of $\Re
^{n}$. In Section 3 and 4 we extend the results, obtained by Marshall and
Olkin \cite{Marshall} into more specific sets of constraints determining
their extremal elements. In Section 5, we provide an application of these
results, dealing with the problem of computing bounds for the second Zagreb
index, $M_{2}(G)$ of a particular class of graphs with a given number of
pendant vertices. This index is extensively studied in graph theory, as a
chemical molecular structure descriptor (\cite{Gutman2}, \cite{Gutman3}, 
\cite{Gutman4}, \cite{Nikolic} and \cite{Todeschini}) and, more generally,
in network analysis, as a measure of degree-assortativity, quantifying how
well a network is connected, (\cite{Alderson}, \cite{Li} and \cite{Li2}). In
the latter context the Zagreb index $M_{2}(G)$ is renamed $S(G).$ Since
determining $S(G)$ requires a specific algorithm (\cite{Li}), many bounds
have been proposed in the literature (\cite{Bollobas}, \cite{Das-Gut-09}, 
\cite{Lu}, \cite{Yan} and \cite{Zhao}). Recently Grassi et al. in \cite%
{Grassi} obtained different bounds through a majorization technique. Using
this approach, we derive new bounds in terms of graph degree sequence and
present some theoretical and numerical examples comparing our results with
the literature. Our conclusions are presented in Section 6. %
%
%

\section{Notations and preliminaries}

Let $\mathbf{e}^{\mathbf{j}},$ $j=1,...n,$ be the fundamental vectors of $%
\mathbb{R}
^{n}$ and set:%
\begin{align*}
\mathbf{s}^{\mathbf{0}}=\mathbf{0},& \text{ }\mathbf{s}^{\mathbf{j}%
}=\dsum\limits_{i=1}^{j}\mathbf{e}^{\mathbf{i}},\text{ }j=1,...n, \\
\mathbf{v}^{\mathbf{n}}=\mathbf{0},& \text{ }\mathbf{v}^{\mathbf{j}%
}=\dsum\limits_{i=j+1}^{n}\mathbf{e}^{\mathbf{i}},\text{ }j=0,...(n-1).
\end{align*}%
Recalling that the Hadamard product of two vectors $\mathbf{x,y}\in \mathbb{R%
}^{n}$ is defined as follows: 
\begin{equation*}
\mathbf{x}\circ \mathbf{y}=\left[ x_{1}y_{1},x_{2}y_{2},...,x_{n}y_{n}\right]
^{T}
\end{equation*}%
it is easy to verify the following properties, where $\left\langle \cdot
,\cdot \right\rangle $ denotes the inner product in $%
\mathbb{R}
^{n}$:

\begin{itemize}
\item[i)] $\left\langle \mathbf{x\circ y,z}\right\rangle =\left\langle 
\mathbf{x,y\circ z}\right\rangle $

\item[ii)] $\langle \mathbf{s^{h}},\mathbf{v}^{\mathbf{k}}\rangle =h-\min
\left\{ h,k\right\} $

\item[iii)] $\mathbf{s}^{\mathbf{k}}\circ \mathbf{s}^{\mathbf{j}}=\mathbf{%
s^{h}},$ $h=\min \left\{ k,j\right\} $

\item[iv)] $\mathbf{v}^{\mathbf{k}}\circ \mathbf{s}^{\mathbf{j}}=\mathbf{s}^{%
\mathbf{j}}-\mathbf{s^{h}}=\mathbf{v}^{\mathbf{h}}-\mathbf{v}^{\mathbf{j}},$ 
$h=\min \left\{ k,j\right\} $
\end{itemize}

\begin{definition}
Assuming that the components of the vectors $\mathbf{x}$, $\mathbf{y\in }$ $%
\mathbb{R}
^{n}$ are arranged in nonincreasing order, the majorization order $\mathbf{x}%
\trianglelefteq \mathbf{y}$ means:%
\begin{equation*}
\left\langle \mathbf{x},\mathbf{s}^{\mathbf{k}}\right\rangle \leq
\left\langle \mathbf{y},\mathbf{s}^{\mathbf{k}}\right\rangle ,\text{ }%
k=1,...,(n-1)
\end{equation*}%
and%
\begin{equation*}
\left\langle \mathbf{x},\mathbf{s}^{\mathbf{n}}\right\rangle =\left\langle 
\mathbf{y},\mathbf{s}^{\mathbf{n}}\right\rangle .
\end{equation*}
\end{definition}

\noindent In the sequel $\mathbf{x}^{\ast }(S)$ and $\mathbf{x}_{\ast }(S)$
will denote the maximal and the minimal elements of a subset $S\subseteq 
\mathbb{R}^{n}$ with respect to the majorization order.

\noindent Given a positive real number $a$, it is well known \cite{Marshall}
that the maximal and the minimal elements of the set 
\begin{equation*}
\Sigma _{a}=\left\{ \mathbf{x\in }\text{ }%
\mathbb{R}
^{n}:x_{1}\geq x_{2}\geq ...\geq x_{n}\geq 0,\left\langle \mathbf{x},\mathbf{%
s}^{\mathbf{n}}\right\rangle =a\right\}
\end{equation*}%
with respect to the majorization order are respectively%
\begin{equation*}
\mathbf{x}^{\ast }\left( \Sigma _{a}\right) =a\mathbf{e}^{1}\text{ \ \ \ \ \
\ \ \ \ and \ \ \ \ \ \ \ \ \ \ }\mathbf{x}_{\ast }\left( \Sigma _{a}\right)
=\frac{a}{n}\mathbf{s}^{\mathbf{n}}.
\end{equation*}

\noindent Next sections are dedicated to the study of the maximal and the
minimal elements, with respect to the majorization order, of the particular
subset of $\Sigma _{a}$ given by 
\begin{equation}
S_{a}=\Sigma _{a}\cap \left\{ \mathbf{x\in }\text{ }%
\mathbb{R}
^{n}:M_{i}\geq x_{i}\geq m_{i},\text{ }i=1,...n\right\} ,  \label{Sa}
\end{equation}%
where $\mathbf{m}=\left[ m_{1},m_{2},...,m_{n}\right] ^{T}$ and $\mathbf{M}=%
\left[ M_{1},M_{2},...,M_{n}\right] ^{T}$ are two assigned vectors arranged
in nonincreasing order with $0\leq m_{i}\leq M_{i},$ for all $i=1,...n,$ and 
$a$ is a positive real number such that $\left\langle \mathbf{m},\mathbf{s}^{%
\mathbf{n}}\right\rangle \leq a\leq \left\langle \mathbf{M},\mathbf{s}^{%
\mathbf{n}}\right\rangle .$ Notice that the intervals $\left[ m_{i},M_{i}%
\right] $ are not necessarily disjointed unless the additional assumption $%
M_{i+1}<m_{i},$ $i=1,...,(n-1)$ is required. The existence of maximal and
minimal elements of $S_{a}$ are ensured by the compactness of the set $S_{a}$
and by the closure of the upper and level sets:%
\begin{equation*}
U(\mathbf{x})=\left\{ \mathbf{z\in }\text{ }S_{a}:\mathbf{x\trianglelefteq z}%
\right\} ,\text{ }L(\mathbf{x})=\left\{ \mathbf{z\in }\text{ }S_{a}:\mathbf{%
z\trianglelefteq x}\right\}
\end{equation*}


%
%
%
%
%
%

\section{The maximal element of $S_{a}$}

We start computing the maximal element, with respect to the majorization
order, of the set $S_a$.


\begin{theorem}
Let $k\geq 0$ be the smallest integer such that 
\begin{equation}
\left\langle \mathbf{M},\mathbf{s}^{\mathbf{k}}\right\rangle +\left\langle 
\mathbf{m},\mathbf{v}^{\mathbf{k}}\right\rangle \leq a<\left\langle \mathbf{M%
},\mathbf{s}^{k+1}\right\rangle +\left\langle \mathbf{m},\mathbf{v}^{\mathbf{%
k+1}}\right\rangle ,  \label{dis1}
\end{equation}%
and $\theta =a-\left\langle \mathbf{M},\mathbf{s}^{\mathbf{k}}\right\rangle
-\left\langle \mathbf{m},\mathbf{v}^{\mathbf{k+1}}\right\rangle .$ Then 
\begin{equation}
\mathbf{x}^{\ast }(S_{a})=\mathbf{M\circ s}^{\mathbf{k}}+\theta \mathbf{e}%
^{k+1}+\mathbf{m\circ v}^{\mathbf{k+1}}  \label{xsa}
\end{equation}
\end{theorem}

\begin{proof}
First of all we verify that $\mathbf{x}^{\ast }\mathbf{(}S_{a}\mathbf{)}\in $
$S_{a}$. It easy to see that $\langle \mathbf{x}^{\ast }\mathbf{(}S_{a}%
\mathbf{)},$\textbf{$s^{n}$}$\rangle =a$ and that $m_{i}\leq \mathbf{x}%
_{i}^{\ast }\mathbf{(}S_{a}\mathbf{)\leq }M_{i}$ for $i\neq k+1.$ To prove
that $m_{k+1}\leq \mathbf{x}_{k+1}^{\ast }\mathbf{(}S_{a}\mathbf{)\leq }%
M_{k+1}$, notice that from (\ref{dis1}) 
\begin{equation*}
m_{k+1}=\left\langle \mathbf{m,e}^{\mathbf{k+1}}\right\rangle \leq
a-\left\langle \mathbf{M},\mathbf{s}^{\mathbf{k}}\right\rangle -\left\langle 
\mathbf{m},\mathbf{v}^{\mathbf{k+1}}\right\rangle =\theta <\left\langle 
\mathbf{M},\mathbf{e}^{\mathbf{k+1}}\right\rangle =M_{k+1}.
\end{equation*}

\noindent Now we show that $\mathbf{x}\trianglelefteq \mathbf{x}^{\ast
}(S_{a})$ for all $\mathbf{x\in }$ $S_{a}.$ By property i) follows 
\begin{equation*}
\left\langle \mathbf{x}^{\ast }(S_{a}),\mathbf{s}^{\mathbf{j}}\right\rangle
=\left\langle \mathbf{M},\mathbf{s}^{\mathbf{k}}\circ \mathbf{s}^{\mathbf{j}%
}\right\rangle +\theta \left\langle \mathbf{e}^{\mathbf{k+1}},\mathbf{s}^{%
\mathbf{j}}\right\rangle +\left\langle \mathbf{m},\mathbf{v}^{\mathbf{k+1}%
}\circ \mathbf{s}^{\mathbf{j}}\right\rangle ,\text{ }j=1,...(n-1)
\end{equation*}%
and by iii) and iv) 
\begin{equation*}
\left\langle \mathbf{x}^{\ast }(S_{a}),\mathbf{s}^{\mathbf{j}}\right\rangle
=\left\{ 
\begin{array}{cc}
\left\langle \mathbf{M},\mathbf{s}^{\mathbf{j}}\right\rangle & 1\leq j\leq k
\\ 
\left\langle \mathbf{M},\mathbf{s}^{\mathbf{k}}\right\rangle +\theta
+\left\langle \mathbf{m},\mathbf{s}^{\mathbf{j}}-\mathbf{s}^{\mathbf{k+1}%
}\right\rangle & (k+1)\leq j\leq (n-1)%
\end{array}%
\right. .
\end{equation*}

\noindent Thus, given a vector $\mathbf{x\in }$ $S_{a}$, for $1\leq j\leq k$
we obtain 
\begin{equation*}
\left\langle \mathbf{x},\mathbf{s}^{\mathbf{j}}\right\rangle \leq
\left\langle \mathbf{M},\mathbf{s}^{\mathbf{j}}\right\rangle =\left\langle 
\mathbf{x}^{\ast }(S_{a}),\mathbf{s}^{\mathbf{j}}\right\rangle ,
\end{equation*}%
while for $(k+1)\leq j\leq (n-1)$, by iii), 
\begin{equation*}
\left\langle \mathbf{x},\mathbf{s}^{\mathbf{j}}\right\rangle =\left\langle 
\mathbf{x},\mathbf{\mathbf{s}^{\mathbf{n}}}\right\rangle -\left\langle 
\mathbf{x},\mathbf{v}^{\mathbf{j}}\right\rangle \leq a-\left\langle \mathbf{m%
},\mathbf{v}^{\mathbf{j}}\right\rangle =\left\langle \mathbf{M},\mathbf{s}^{%
\mathbf{k}}\right\rangle +\theta +\left\langle \mathbf{m},\mathbf{s}^{%
\mathbf{j}}-\mathbf{s}^{\mathbf{k+1}}\right\rangle =\left\langle \mathbf{x}%
^{\ast }(S_{a}),\mathbf{s}^{\mathbf{j}}\right\rangle
\end{equation*}%
and the result follows.
\end{proof}

\bigskip

\noindent From this general result, the maximal element of particular
subsets of $S_{a}$ can be deduced. We then focus on a specific case which
will be useful in the application we deal with in Section 5. We denote by $%
\left\lfloor x\right\rfloor $ the integer part of the real number $x$.


\begin{corollary}
\label{cor:dueintervalli} Given $1\leq h\leq n$, let us consider the set 
\begin{equation}
S_{a}^{\left[ h\right] }=%
\begin{array}{c}
\Sigma _{a}\cap \left\{ \mathbf{x}\in \mathbb{R}^{n}:M_{1}\geq x_{1}\geq
...\geq x_{h}\geq m_{1},\right.  \\ 
\left. M_{2}\geq x_{h+1}\geq ...\geq x_{n}\geq m_{2}\right\} 
\end{array}%
,  \label{Sh}
\end{equation}%
where 
\begin{equation*}
0\leq m_{2}\leq m_{1},0\leq M_{2}\leq M_{1},m_{i}<M_{i},i=1,2
\end{equation*}%
and 
\begin{equation*}
hm_{1}+(n-h)m_{2}\leq a\leq hM_{1}+(n-h)M_{2}.
\end{equation*}

\noindent Let $a^{\ast }=hM_{1}+(n-h)m_{2}$ and 
\begin{equation*}
k=\left\{ 
\begin{array}{ccc}
\left\lfloor \dfrac{a-h(m_{1}-m_{2})-nm_{2}}{M_{1}-m_{1}}\right\rfloor  & 
\text{ if } & a<a^{\ast } \\ 
&  &  \\ 
\left\lfloor \dfrac{a-h(M_{1}-M_{2})-nm_{2}}{M_{2}-m_{2}}\right\rfloor  & 
\text{ if } & a\geq a^{\ast }%
\end{array}%
\right. 
\end{equation*}%
Then 
\begin{equation*}
\mathbf{x}^{\ast }(S_{a}^{\left[ h\right] })=\left\{ 
\begin{array}{ccc}
\left[ \underset{k}{\underbrace{M_{1},.....,M_{1}}},\theta ,\underset{h-k-1}{%
\underbrace{m_{1},.....,m_{1}},}\underset{n-h}{\underbrace{m_{2},.....,m_{2}}%
}\right]  & \text{ if } & a<a^{\ast } \\ 
&  &  \\ 
\left[ \underset{h}{\underbrace{M_{1},.....,M_{1}}},\underset{k-h}{%
\underbrace{M_{2},.....,M_{2}}},\theta ,\underset{n-k-1}{\underbrace{%
m_{2},.....m_{2}}}\right]  & \text{ if } & a\geq a^{\ast }%
\end{array}%
\right. 
\end{equation*}%
where $\mathbf{M}=M_{1}\mathbf{s^{h}}+M_{2}\mathbf{v^{h}},$ $\mathbf{m}=m_{1}%
\mathbf{s^{h}}+m_{2}\mathbf{v^{h}}$ and $\theta =a-\left\langle \mathbf{M},%
\mathbf{s}^{\mathbf{k}}\right\rangle -\left\langle \mathbf{m},\mathbf{v}^{%
\mathbf{k+1}}\right\rangle $.
\end{corollary}

\begin{proof}
Easy computations give: 
\begin{equation*}
\left\langle \mathbf{M},\mathbf{s}^{\mathbf{k}}\right\rangle =\left\{ 
\begin{array}{ccc}
hM_{1}+M_{2}\left( k-h\right) & \text{ if } & k\geq h \\ 
kM_{1} & \text{ if } & k<h%
\end{array}%
\right.
\end{equation*}%
\begin{equation*}
\left\langle \mathbf{m},\mathbf{v}^{\mathbf{k}}\right\rangle =\left\{ 
\begin{array}{ccc}
\left( n-k\right) m_{2} & \text{ if } & k\geq h \\ 
\left( n-h\right) m_{2}+m_{1}\left( h-k\right) & \text{ if } & k<h%
\end{array}%
\right.
\end{equation*}%
and the values are linked for continuity when $k=h. $ We distinguish two
cases:

\begin{enumerate}
\item[i)] $k\geq h:$ from (\ref{dis1}) we have 
\begin{equation*}
k=\left\lfloor \dfrac{a-h(M_{1}-M_{2})-nm_{2}}{M_{2}-m_{2}}\right\rfloor 
\end{equation*}%
that is acceptable only if $a\geq hM_{1}+(n-h)m_{2}=a^{\ast }.$ Then, from (%
\ref{xsa}) 
\begin{equation*}
\mathbf{x}^{\ast }(S_{a}^{\left[ h\right] })=\left[ \underset{h}{\underbrace{%
M_{1},.....,M_{1}}},\underset{k-h}{\underbrace{M_{2},.....,M_{2}}},\theta ,%
\underset{n-k-1}{\underbrace{m_{2},.....m_{2}}}\right] .
\end{equation*}

\item[ii)] $k<h:$ from (\ref{dis1}) we get%
\begin{equation*}
k=\left\lfloor \dfrac{a-h(m_{1}-m_{2})-nm_{2}}{M_{1}-m_{1}}\right\rfloor 
\end{equation*}%
that is acceptable only if $a<hM_{1}+(n-h)m_{2}=a^{\ast }.$ Then, from (\ref%
{xsa})%
\begin{equation*}
\mathbf{x}^{\ast }(S_{a}^{\left[ h\right] })=\left[ \underset{k}{\underbrace{%
M_{1},.....,M_{1}}},\theta ,\underset{h-k-1}{\underbrace{m_{1},.....,m_{1}},}%
\underset{n-h}{\underbrace{m_{2},.....,m_{2}}}\right] .
\end{equation*}
\end{enumerate}
\end{proof}


\begin{remark}
\label{rem:a=a*} When $a=a^{\ast }$ it is worthwhile to note that $k=h$ and $%
\theta =m_{2}$ so that 
\begin{equation*}
\mathbf{x}^{\ast }(S_{a}^{\left[ h\right] })=\left[ \underset{h}{\underbrace{%
M_{1},.....,M_{1}}},\underset{n-h}{\underbrace{m_{2},.....,m_{2}}}\right] .
\end{equation*}
\end{remark}


\begin{remark}
\label{rem:m=M} The assumption $m_{i}<M_{i}$ in Corollary \ref%
{cor:dueintervalli} can be relaxed to $m_{i}\leq M_{i}$. Indeed if $%
m_{i}=M_{i},i=1,2$, the set $S_{a}^{\left[ h\right] }$ reduces to the
singleton $\{m_{1}\mathbf{s^{h}}+m_{2}\mathbf{v^{h}}\}$, while if $%
m_{1}=M_{1},m_{2}<M_{2}$ the first $h$ components of any $\mathbf{x}\in
S_{a}^{\left[ h\right] }$ are fixed and equal to $m_{1}$ and the maximal
element of $S_{a}^{\left[ h\right] }$ can be computed by the maximal element
of $S_{a-hm_{1}}\in \mathbb{R}^{n-h}$ (see Corollary \ref{cor:marshall}
below). The case $m_{2}=M_{2},m_{1}<M_{1}$ is similar.
\end{remark}

\noindent The next proposition is proved in \cite{Marshall} and it
immediately follows from Corollary 3 when $m_{1}=m_{2}=m$ and $M_{1}=M_{2}=M$%
.

\begin{corollary}
\label{cor:marshall} Let $0\leq m<M$ and $m\leq \dfrac{a}{n}\leq M.$ Given
the subset 
\begin{equation*}
S_{a}^{1}=\Sigma _{a}\cap \left\{ \mathbf{x\in }\text{ }%
\mathbb{R}
^{n}:M\geq x_{1}\geq x_{2}\geq ...\geq x_{n}\geq m\right\}
\end{equation*}%
we have 
\begin{equation*}
\mathbf{x}^{\ast }(S_{a}^{1})=M\mathbf{s}^{\mathbf{k}}+\theta \mathbf{e}^{%
\mathbf{k+1}}+m\mathbf{v}^{\mathbf{k+1}},
\end{equation*}%
where $k=\left\lfloor \dfrac{a-nm}{M-m}\right\rfloor $ and $\theta
=a-Mk-m\left( n-k-1\right) .$

\noindent In particular when $m=0$ we obtain 
\begin{equation*}
\mathbf{x}^{\ast }(S_{a}^{1})=M\mathbf{s}^{\mathbf{k}}+\theta \mathbf{e}^{%
\mathbf{k+1}},
\end{equation*}%
where $k=\left\lfloor \dfrac{a}{M}\right\rfloor $ and $\theta =a-Mk$.
\end{corollary}

\noindent It is worthwhile to notice that $S_{a}$, is a subset of $S_{a}^{1}$
where $m=m_{n}$ and $M=M_{1}$. Thus the following inequality holds: 
\begin{equation}
x^{\ast }(S_{a})\text{ }\trianglelefteq \text{ }x^{\ast }(S_{a}^{1}).
\label{confronto}
\end{equation}

\noindent Finally we recall the following result (see \cite{Bianchi}).

\begin{corollary}
Let $1\leq h\leq n$ and $0<\alpha \leq a/h.$ Given the subset 
\begin{equation*}
S_{a}^{2}=\Sigma _{a}\cap \left\{ \mathbf{x\in }\text{ }%
\mathbb{R}
^{n}:x_{i}\geq \alpha ,\text{ }i=1,...h\right\} ,
\end{equation*}%
we have $\mathbf{x}^{\ast }(S_{a}^{2})=\left( a-h\alpha \right) \mathbf{e}%
^{1}+\alpha \mathbf{s^{h}}.$
\end{corollary}

\begin{proof}
The set $S_{a}^{2}$ can be obtained by (\ref{Sa}) for $m_{1}=\alpha $, $%
m_{2}=0$, $M_{1}=M_{2}=a.$ Since $a^{\ast }=ha\geq a,$ two cases can be
distinguished:

\begin{enumerate}
\item[i)] $h=1:$ we have $a^{\ast }=a$ and from Remark \ref{rem:a=a*} it
immediately follows that $k=1$ and $\theta =0$ so that 
\begin{equation*}
\mathbf{x}^{\ast }(S_{a}^{2})=a\mathbf{e}^{1}.
\end{equation*}

\item[ii)] $h>1:$ we have $a^{\ast }>a$ and Corollary \ref{cor:dueintervalli}
implies that $k=\left\lfloor \dfrac{a-h\alpha }{a-\alpha }\right\rfloor =0.$
Thus 
\begin{equation*}
\mathbf{x}^{\ast }(S_{a}^{2})=\left[ \theta ,\underset{h-1}{\underbrace{%
\alpha ,...,\alpha }},\underset{n-h}{\underbrace{0,...,0}}\right]
\end{equation*}%
where $\theta =a-(h-1)\alpha $, which leads to 
\begin{equation*}
\mathbf{x}^{\ast }(S_{a}^{2})=\theta \mathbf{e^{1}}+\alpha \mathbf{s^{h}}%
-\alpha \mathbf{e^{1}}=\left( a-h\alpha \right) \mathbf{e}^{1}+\alpha 
\mathbf{s^{h}}.
\end{equation*}
\end{enumerate}
\end{proof}

\section{The minimal element of $S_{a}$}

In this section we study the structure of the minimal element, with respect
to the majorization order, of the set $S_a$.

\begin{theorem}
\label{th:minimo} \label{minimo} Let $k \ge 0$ and $d \ge 0$ be the smallest
integers such that

\begin{itemize}
\item[1)] $k + d <n $

\item[2)] $m_{k+1} \le \rho \le M_{n-d}$ where $\rho= \dfrac { a - \langle 
\mathbf{m},\mathbf{s^k} \rangle - \langle \mathbf{M}, \mathbf{v^{n-d}}
\rangle} {n-k-d}$.
\end{itemize}

Then 
\begin{equation*}
\mathbf{x_{\ast }}(S_{a})=\mathbf{m\circ s^{k}}+\rho (\mathbf{s^{n-d}}-%
\mathbf{s^{k}})+\mathbf{M\circ v^{n-d}}.
\end{equation*}
\end{theorem}

\begin{proof}
The minimal element of the set $\Sigma _{a}$ is $\mathbf{x_{\circ }}(\Sigma
_{a})=\frac{a}{n}\mathbf{s^{n}}$. If $m_{1}\leq \mathbf{x_{\ast }}(\Sigma
_{a})\leq M_{n}$, then $\mathbf{x_{\ast }}(\Sigma _{a})\in S_{a}$ and $%
\mathbf{x_{\ast }}(S_{a})=\mathbf{x_{\ast }}(\Sigma _{a})$ (notice that in
this case $k=d=0$).

\noindent If $\mathbf{x_{\ast }}(\Sigma _{a})\notin S_{a}$, let $k$ and $d$
the smallest integers satisfying conditions 1) and 2) above. It is easy to
verify that $\mathbf{x_{\ast }}(S_{a})\in S_{a}.$ In order to prove that it
is the minimal element, we must show that for all $\mathbf{x}\in S_{a}$ 
\begin{equation}
\langle \mathbf{x_{\ast }}(S_{a}),\mathbf{s^{h}}\rangle \leq \langle \mathbf{%
x},\mathbf{s^{h}}\rangle \,\,\,,h=1,\cdots (n-1).  \label{elemento minimo}
\end{equation}

\noindent We distinguish three cases:

\begin{itemize}
\item[i)] $\mathbf{1\leq h\leq k}$. Since $\langle \mathbf{x_{\ast }}(S_{a}),%
\mathbf{s^{h}}\rangle =\langle \mathbf{m},\mathbf{s^{h}}\rangle $, the
inequality (\ref{elemento minimo}) is straightforward.

\item[ii)] $\mathbf{k+1 \le h \le n-d}$. We prove the inequality (\ref%
{elemento minimo}) for $h=k+1$. By induction, similar arguments can be
applied to prove the inequality for $h=k+2, \cdots (n-d)$.

By contradiction, let us assume that there exists $\mathbf{x}\in S_{a}$ such
that 
\begin{equation*}
\langle \mathbf{x_{\ast }(S_{a})},\mathbf{s^{k+1}}\rangle =\langle \mathbf{m}%
,\mathbf{s^{k}}\rangle +\rho >\langle \mathbf{x},\mathbf{s^{k}}\rangle
+x_{k+1}.
\end{equation*}%
Then $x_{j}\leq x_{k+1}<\langle \mathbf{m},\mathbf{s^{k}}\rangle +\rho
-\langle \mathbf{x},\mathbf{s^{k}}\rangle $, for $j=k+2,\cdots n$ and thus 
\begin{align*}
\langle \mathbf{x},\mathbf{s^{n-d}}\rangle & =\langle \mathbf{x},\mathbf{%
s^{k}}\rangle +\langle \mathbf{x},\mathbf{s^{n-d}}-\mathbf{s^{k}}\rangle < \\
& <\langle \mathbf{x},\mathbf{s^{k}}\rangle +(n-d-k)(\langle \mathbf{m},%
\mathbf{s^{k}}\rangle +\rho -\langle \mathbf{x},\mathbf{s^{k}}\rangle ).
\end{align*}%
Taking into account that 
\begin{equation*}
\langle \mathbf{x},\mathbf{s^{n-d}}\rangle =a-\langle \mathbf{x},\mathbf{%
v^{n-d}}\rangle \geq a-\langle \mathbf{M},\mathbf{v^{n-d}}\rangle ,
\end{equation*}%
we get 
\begin{equation*}
a-\langle \mathbf{M},\mathbf{v^{n-d}}\rangle <(1-n+d+k)\langle \mathbf{x},%
\mathbf{s^{k}}\rangle +(n-d-k)(\langle \mathbf{m},\mathbf{s^{k}}\rangle
+\rho ).
\end{equation*}%
Using the expression of $\rho $, we obtain 
\begin{equation*}
0<(1-n+d+k)(\langle \mathbf{x},\mathbf{s^{k}}\rangle -\langle \mathbf{m},%
\mathbf{s^{k}}\rangle ).
\end{equation*}%
Since $(1-n+d+k)\leq 0$ and $\langle \mathbf{x},\mathbf{s^{k}}\rangle \geq
\langle \mathbf{m},\mathbf{s^{k}}\rangle $, the inequality above is false,
and we have got the contradiction.

\item[iii)] $\mathbf{n-d+1\leq h<n}$. For any $\mathbf{x}\in S_{a}$ we have 
\begin{align*}
\langle \mathbf{x_{\ast }}(S_{a}),\mathbf{s^{h}}\rangle & =\langle \mathbf{%
x_{\ast }}(S_{a}),\mathbf{s^{n-d}}\rangle +\langle \mathbf{x_{\ast }}(S_{a}),%
\mathbf{s^{h}}-\mathbf{s^{n-d}}\rangle = \\
& =\langle \mathbf{m},\mathbf{s^{k}}\rangle +(n-d-k)\rho +\langle \mathbf{M},%
\mathbf{s^{h}}-\mathbf{s^{n-d}}\rangle = \\
& =a-\langle \mathbf{M},\mathbf{v^{n-d}}\rangle +\langle \mathbf{M},\mathbf{%
s^{h}}-\mathbf{s^{n-d}}\rangle = \\
& =a-\langle \mathbf{M},\mathbf{s^{n}}-\mathbf{s^{h}}\rangle = \\
& =\langle \mathbf{x},\mathbf{s^{h}}\rangle +\langle \mathbf{x},\mathbf{s^{n}%
}-\mathbf{s^{h}}\rangle -\langle \mathbf{M},\mathbf{s^{n}}-\mathbf{s^{h}}%
\rangle \\
& \leq \langle \mathbf{x},\mathbf{s^{h}}\rangle .
\end{align*}
\end{itemize}
\end{proof}

\noindent Now we analyze the minimal element of particular subsets of $S_{a}$%
. We start considering the intervals $[m_{i},M_{i}],i=1,\cdots ,n$
disjointed. Notice that this additional assumption does not modify the
choice of the maximal element, while it simplifies the choice of the minimal
element.

\begin{corollary}
Let us consider the set $S_{a}$ and assume 
\begin{equation}
M_{i+1}<m_{i} \text{ for } i=1,...(n-1).  \label{eq:disgiunto}
\end{equation}%
Let $k\geq 0$ be the smallest integer such that 
\begin{equation}
\left\langle \mathbf{m},\mathbf{s}^{\mathbf{k+1}}\right\rangle +\left\langle 
\mathbf{M},\mathbf{v}^{\mathbf{k+1}}\right\rangle \leq a<\left\langle 
\mathbf{m},\mathbf{s}^{\mathbf{k}}\right\rangle +\left\langle \mathbf{M},%
\mathbf{v}^{\mathbf{k}}\right\rangle  \label{bound}
\end{equation}%
and $\rho =a-\left\langle \mathbf{m},\mathbf{s}^{\mathbf{k}}\right\rangle
-\left\langle \mathbf{M},\mathbf{v}^{\mathbf{k+1}}\right\rangle .$ Then 
\begin{equation*}
\mathbf{x}_{\ast }(S_{a})=\mathbf{m\circ s}^{\mathbf{k}}+\rho \mathbf{e}^{%
\mathbf{k+1}}+\mathbf{M\circ v}^{\mathbf{k+1}}
\end{equation*}
\end{corollary}

\begin{proof}
By condition 2) in Theorem \ref{th:minimo} and assumption (\ref{eq:disgiunto}%
), we get 
\begin{equation*}
M_{k+2}<m_{k+1}\leq \rho \leq M_{n-d}.
\end{equation*}%
Thus $k>n-d-2$. Since $k$ is an integer such that $k<n-d$, we have
necessarily $k=n-d-1$ and the thesis follows.
\end{proof}


\bigskip

\noindent Another case of practical interest regards the set studied in
Corollary \ref{cor:dueintervalli}.

\begin{corollary}
\label{dueintervalliminimo} Given $1\leq h\leq n$, let us consider the set 
\begin{equation*}
S_{a}^{\left[ h\right] }=%
\begin{array}{c}
\Sigma _{a}\cap \left\{ \mathbf{x}\in \mathbb{R}^{n}:M_{1}\geq x_{1}\geq
...\geq x_{h}\geq m_{1},\right.  \\ 
\left. M_{2}\geq x_{h+1}\geq ...\geq x_{n}\geq m_{2}\right\} 
\end{array}%
,
\end{equation*}%
where $0\leq m_{2}\leq m_{1}$, $0\leq M_{2}\leq M_{1}$, $m_{i}<M_{i},i=1,2$
and 
\begin{equation*}
hm_{1}+(n-h)m_{2}\leq a\leq hM_{1}+(n-h)M_{2}.
\end{equation*}%
If $m_{1}\leq \dfrac{a}{n}\leq M_{2}$ we have $x_{\ast }(S_{a}^{\left[ h%
\right] })=\frac{a}{n}\mathbf{s^{n}}$. Otherwise, let $\widetilde{a}%
=hm_{1}+(n-h)M_{2}$. If $\left\{ 
\begin{array}{c}
a<m_{1}n \\ 
a\leq \widetilde{a}%
\end{array}%
\right. $, given $\rho =\dfrac{a-hm_{1}}{n-h}$, we have 
\begin{equation*}
\mathbf{x_{\ast }}(S_{a}^{\left[ h\right] })=m_{1}\mathbf{s^{h}}+\rho 
\mathbf{v^{h}}=\left[ \underset{h}{\underbrace{m_{1},.....,m_{1}}},\underset{%
n-h}{\underbrace{\rho ,.....,\rho }}\right] .
\end{equation*}%
\noindent If $\left\{ 
\begin{array}{c}
a>M_{2}n \\ 
a\geq \widetilde{a}%
\end{array}%
\right. $ , given $\rho =\dfrac{a-M_{2}(n-h)}{h},$ we have 
\begin{equation*}
\mathbf{x_{\ast }}(S_{a}^{\left[ h\right] })=\rho \mathbf{s^{h}}+M_{2}%
\mathbf{v^{h}}=\left[ \underset{h}{\underbrace{\rho ,...,\rho }},\underset{%
n-h}{\underbrace{M_{2},...,M_{2}}}\right] .
\end{equation*}
\end{corollary}

\begin{proof}
Let us investigate when the best choice $k=d=0$ is admissible. Under this
assumption, from condition $2)$ in Theorem \ref{minimo} we have 
\begin{equation}
m_{1}\leq \rho =\dfrac{a}{n}\leq M_{n}=M_{2}.  \label{disequa1}
\end{equation}%
If the condition above holds, the minimal element is $\mathbf{x_{\ast }}%
(S_{a}^{\left[ h\right] })=\dfrac{a}{n}\mathbf{s}^{\mathbf{n}}.$

\noindent Otherwise if condition (\ref{disequa1}) does not hold, we begin
with the case $k=0$. We have 
\begin{equation*}
\rho =\dfrac{a-<\mathbf{M},\mathbf{v^{n-d}}>}{n-d}
\end{equation*}%
and 
\begin{equation*}
x_{\ast }(S_{a}^{\left[ h\right] })=\rho \mathbf{s^{n-d}}+\mathbf{M\circ
v^{n-d}}.
\end{equation*}%
From condition $2)$ in Theorem \ref{minimo}, we have $m_{1}\leq \rho \leq
M_{n-d}$ and, taking into account that the elements in $x_{\ast }(S_{a}^{%
\left[ h\right] })$ are in nonincreasing order, $\rho \geq M_{n-d+1}$. We
distinguish three cases:

\begin{enumerate}
\item[i)] if $n-d>h$ then necessarily $\rho=M_2$, but this contradicts (\ref%
{disequa1}).

\item[ii)] if $n-d<h$ then $\rho =M_{1}$ and this is admissible only if $%
a=M_{1}h+M_{2}(n-h),$ so that 
\begin{equation*}
\mathbf{x_{\ast }}(S_{a}^{\left[ h\right] })=\left[ \underset{h}{\underbrace{%
M_{1},.....,M_{1}}},\underset{n-h}{\underbrace{M_{2},.....,M_{2}}}\right] .
\end{equation*}

\item[iii)] if $n-d=h,$ then $\rho =\dfrac{a-M_{2}d}{n-d}$ and 
\begin{equation*}
\mathbf{x_{\ast }}(S_{a}^{\left[ h\right] })=\left[ \underset{h}{\underbrace{%
\rho ,.....,\rho }},\underset{n-h}{\underbrace{M_{2},.....,M_{2}}}\right] .
\end{equation*}%
This result is admissible only if $\rho >M_{2}$ and $m_{1}\leq \rho \leq
M_{1}$, i.e. if $\left\{ 
\begin{array}{c}
a>M_{2}n \\ 
a\geq \widetilde{a}%
\end{array}%
\right. .$
\end{enumerate}

\noindent A symmetric case occurs when $d=0,$ so we have 
\begin{equation*}
\rho =\dfrac{a-<\mathbf{m},\mathbf{s^{k}}>}{n-k}
\end{equation*}%
and 
\begin{equation*}
x_{\ast }(S_{a}^{\left[ h\right] })=\mathbf{m\circ s^{k}}+\rho \mathbf{v^{k}}%
.
\end{equation*}%
From condition $2)$ in Theorem \ref{minimo}, we have that $m_{k+1}\leq \rho
\leq M_{2}$ and, taking into account that the elements in $x_{\ast }(S_{a}^{%
\left[ h\right] })$ are in nonincreasing order, $\rho \leq m_{k}$. We
distinguish three cases:

\begin{enumerate}
\item[i)] if $k<h$ then necessarily $\rho=m_1$, but this contradicts (\ref%
{disequa1}).

\item[ii)] if $k>h,$ then $\rho =m_{2}$ and this is possible only if $%
a=hm_{1}+m_{2}(n-h),$ so that 
\begin{equation*}
\mathbf{x_{\ast }}(S_{a}^{\left[ h\right] })=\left[ \underset{h}{\underbrace{%
m_{1},.....,m_{1}}},\underset{n-h}{\underbrace{m_{2},.....,m_{2}}}\right] .
\end{equation*}

\item[iii)] if $k=h$, then $\rho =\dfrac{a-hm_{1}}{n-h}$ and 
\begin{equation*}
\mathbf{x_{\ast }}(S_{a}^{\left[ h\right] })=\left[ \underset{h}{\underbrace{%
m_{1},.....,m_{1}}},\underset{n-h}{\underbrace{\rho ,.....,\rho }}\right] .
\end{equation*}%
This result is admissible only if $m_{2}\leq \rho \leq M_{2}$ and $\rho
<m_{1}$, i.e. only if $\left\{ 
\begin{array}{c}
a<m_{1}n \\ 
a\leq \widetilde{a}%
\end{array}%
.\right. $
\end{enumerate}
\end{proof}

\noindent Corollary \ref{dueintervalliminimo} distinguishes the minimal
element of $S_{a}^{\left[ h\right] }$ whether 
\begin{equation*}
\left\{ 
\begin{array}{c}
a<m_{1}n \\ 
a\leq \widetilde{a}%
\end{array}%
\right. \text{ or }\left\{ 
\begin{array}{c}
a>M_{2}n \\ 
a\geq \widetilde{a}%
\end{array}%
.\right. 
\end{equation*}%
We note that if $m_{1}\leq M_{2}$ the first inequality in the systems above
is always stronger than the second one, while if $M_{2}<m_{1}$ the second
one is stronger than the first. Thus we can summarize the minimal element of 
$S_{a}^{\left[ h\right] }$ in a more accessible way according to the
following scheme:

\begin{enumerate}
\item[i)] If $m_{1}\leq M_{2}$ then 
\begin{equation}
x_{\ast }(S_{a}^{\left[ h\right] })=\left\{ 
\begin{array}{ccc}
\dfrac{a}{n}\mathbf{s^{n}} & \text{ if } & m_{1}\leq \dfrac{a}{n}\leq M_{2}
\\ 
m_{1}\mathbf{s^{h}}+\dfrac{a-hm_{1}}{n-h}\mathbf{v^{h}} & \text{ if } & 
\dfrac{a}{n}<m_{1} \\ 
\dfrac{a-M_{2}(n-h)}{h}\mathbf{s^{h}}+M_{2}\mathbf{v^{h}} & \text{ if } & 
\dfrac{a}{n}>M_{2} \\ 
&  & 
\end{array}%
\right.   \label{m_1 le M_2}
\end{equation}%
and the vectors are linked for continuity.

\item[ii)] If $M_{2}<m_{1}$ then 
\begin{equation}
x_{\ast }(S_{a}^{\left[ h\right] })=\left\{ 
\begin{array}{ccc}
m_{1}\mathbf{s^{h}}+\dfrac{a-hm_{1}}{n-h}\mathbf{v^{h}} & \text{ if } & a<%
\widetilde{a} \\ 
\dfrac{a-M_{2}(n-h)}{h}\mathbf{s^{h}}+M_{2}\mathbf{v^{h}} & \text{ if } & 
a\geq \widetilde{a}. \\ 
&  & 
\end{array}%
\right.   \label{M_2<m_1}
\end{equation}
\end{enumerate}


\begin{remark}
\label{remark:a=atilde}

When $a=\widetilde{a}$ it is worthwhile to note that 
\begin{equation*}
\mathbf{x}_{\ast }(S_{a}^{\left[ h\right] })=m_{1}\mathbf{s^{h}}+M_{2}%
\mathbf{v^{h}}=\left[ \underset{h}{\underbrace{m_{1},.....,m_{1}}},\underset{%
n-h}{\underbrace{M_{2},.....,M_{2}}}\right] .
\end{equation*}
\end{remark}


\begin{remark}
\label{re:componenti intere} We note that the minimal element of the set $%
S_{a}^{\left[ h\right] }$ does not necessarily have integer components,
while this is not the case for the maximal element. For some applications,
it is meaningful to find the minimal vector in $S_{a}^{\left[ h\right] }$
with integer components. We illustrate below the procedure to follow. Let us
consider, for instance, the vector $x_{\ast }(S_{a}^{\left[ h\right] })=%
\dfrac{a}{n}\mathbf{s^{n}}$ which corresponds to the case $m_{1}\leq \dfrac{a%
}{n}\leq M_{2}$ (see (\ref{m_1 le M_2})). If $\dfrac{a}{n}$ is not an
integer, let us find the index $k$, $1\leq k\leq n$, such that 
\begin{equation*}
(\lfloor \dfrac{a}{n}\rfloor +1)k+\lfloor \dfrac{a}{n}\rfloor (n-k)=a
\end{equation*}%
i.e. $k=a-\lfloor \dfrac{a}{n}\rfloor n$. The vector 
\begin{equation*}
\mathbf{x_{\ast }^{1}}=(\lfloor \dfrac{a}{n}\rfloor +1)\mathbf{s^{k}}%
+\lfloor \dfrac{a}{n}\rfloor \mathbf{v^{k}}
\end{equation*}%
is the minimal element of $S_{a}^{\left[ h\right] }$ with integer components.

\noindent With slight modification, the same procedure can be applied also
in the other cases illustrated in (\ref{m_1 le M_2}) or (\ref{M_2<m_1}),
where only some of the components of $x_{\ast }(S_{a}^{\left[ h\right] })$
can be non integer.
\end{remark}

\noindent To complete our analysis, we show how from Corollary \ref%
{dueintervalliminimo}, particular cases can be deduced. More precisely,
assuming in Corollary \ref{dueintervalliminimo} $m_{1}=m_{2}$, $M_{1}=M_{2}$
or $h=n$ we obtain the results proved in \cite{Marshall}.

\begin{corollary}
\label{cor:marshall2} Let $0\leq m<M$ and $m\leq \dfrac{a}{n}\leq M.$ Given
the subset 
\begin{equation*}
S_{a}^{1}=\Sigma _{a}\cap \left\{ \mathbf{x\in }\text{ }%
\mathbb{R}
^{n}:M\geq x_{1}\geq ...\geq x_{n-1}\geq x_{n}\geq m\right\}
\end{equation*}%
we have $x_{\ast }(S_{a}^{1})=\dfrac{a}{n}\mathbf{s^{n}}$.
\end{corollary}


\noindent As we did with the maximal element, it is clear that the vector
provided by Corollary \ref{dueintervalliminimo} majorizes the vector in
Corollary \ref{cor:marshall2}, i.e. the following inequality holds: 
\begin{equation}
x_{\ast }(S_{a}^{1})\text{ }\trianglelefteq \text{ }x_{\ast }(S_{a}).
\label{confronto1}
\end{equation}

\noindent Assuming $m_{1}=\alpha $, $m_{2}=0$, $M_{1}=M_{2}=a$ or $%
m_{1}=m_{2}=0$ and $M_{2}=\alpha $, $M_{1}=a$ we easily obtain the following
two corollaries ( see \cite{Bianchi}).


\begin{corollary}
Let $1\leq h\leq n$ and $0<\alpha \leq a/h.$ Given the subset 
\begin{equation*}
S_{a}^{2}=\Sigma _{a}\cap \left\{ \mathbf{x\in }\text{ }%
\mathbb{R}
^{n}:x_{i}\geq \alpha ,\text{ }i=1,...h\right\} ,
\end{equation*}%
we have 
\begin{equation*}
x_{\ast }(S_{a}^{2})=\left\{ 
\begin{array}{cc}
\dfrac{a}{n}\mathbf{s^{n}} & \text{ if }\alpha \leq \dfrac{a}{n} \\ 
\alpha \mathbf{s^{h}}+\rho \mathbf{v^{h}}\text{ with }\rho =\dfrac{a-\alpha h%
}{n-h} & \text{ if }\alpha >\dfrac{a}{n}%
\end{array}%
\right.
\end{equation*}
\end{corollary}


\begin{corollary}
Let $1\leq h\leq (n-1)$ and $0<\alpha <a.$ Given the subset 
\begin{equation*}
S_{a}^{3}=\Sigma _{a}\cap \left\{ \mathbf{x\in }\text{ }%
\mathbb{R}
^{n}:x_{i}\leq \alpha ,\text{ }i=h+1,...n\right\} ,
\end{equation*}%
we have 
\begin{equation*}
x_{\ast }(S_{a}^{3})=\left\{ 
\begin{array}{cc}
\dfrac{a}{n}\mathbf{s^{n}} & \text{ if }\alpha \geq \dfrac{a}{n} \\ 
\rho \mathbf{s^{h}}+\alpha \mathbf{v^{h}}\text{ with }\rho =\dfrac{%
a-(n-h)\alpha }{h} & \text{ if }\alpha <\dfrac{a}{n}%
\end{array}%
\right.
\end{equation*}
\end{corollary}


\section{New bounds for the second Zagreb index}

\noindent Let $G=(V,E)$ a simple, connected, undirected graph with fixed
order $\left\vert V\right\vert =n\ $ and fixed size $\left\vert E\right\vert
=m.$ Denote by $\pi =(d_{1},d_{2},..,d_{n})$ the degree sequence of $G,$
being $d_{i}$ the degree of vertex $v_{i}$, arranged in nonincreasing order $%
d_{1}\geq d_{2}\geq \cdots \geq d_{n}$. We recall that the sequences of
integers which are degree sequences of a simple graph were characterized by
Erd\"{o}s and Gallay (see \cite{Erdos}). The second Zagreb index is defined
as 
\begin{equation*}
S(G)\underset{\left( v_{i},v_{j}\right) \in E}{=\sum d_{i}d_{j}}
\end{equation*}%
or equivalently (\cite{Grassi}) 
\begin{equation}
S(G)=\dfrac{\underset{\left( v_{i},v_{j}\right) \in E}{\sum }\left(
d_{i}+d_{j}\right) ^{2}-\overset{n}{\underset{i=1}{\sum }}d_{i}^{3}}{2}.
\label{SG}
\end{equation}%
\noindent In order to compute upper and lower bounds for $S(G)$ we refer to 
\cite{Grassi}, where a methodology based on majorization order was proposed.
Before presenting our results, we briefly describe the procedure we will
follow.

\noindent\ Let $\pi $ be a fixed degree sequence and $\mathbf{x}\in \mathbb{R%
}^{m}$ the vector whose components are $d_{i}+d_{j}$, $(v_{i},v_{j})\in E$.
In \cite{Lovasz} it is shown that 
\begin{equation*}
\underset{\left( v_{i},v_{j}\right) \in E}{\sum }\left( d_{i}+d_{j}\right) =%
\overset{n}{\underset{i=1}{\sum }}d_{i}^{2}
\end{equation*}%
and thus $\sum_{i=1}^{m}x_{i}=\sum_{i=1}^{n}d_{i}^{2}$ is a constant. Given
a suitable subset $S$ of 
\begin{equation*}
\Sigma _{a}=\left\{ \mathbf{x\in \mathbb{R}}^{m}:x_{1}\geq x_{2}\geq ...\geq
x_{m}\geq 0,\sum_{i=1}^{m}x_{i}=a\right\} ,
\end{equation*}%
where $a=\sum_{i=1}^{n}d_{i}^{2},$ the Schur-convex function $f(\mathbf{x})=%
\overset{m}{\underset{i=1}{\sum }}x_{i}^{2}$ attains its minimum and maximum
on $S$ at $f(x_{\ast }(S))$ and $f(x^{\ast }(S))$ respectively$,$ being $%
x_{\ast }(S)$ and $x^{\ast }(S)$ the extremal vectors of $S$ with respect to
the majorization order (see \cite{Marshall}). Hence from (\ref{SG}) the
maximum and the minimum of $S(G)$ can be easily deduced. 


\noindent Let $C_{\pi }$ be the class of graphs $G=(V,E)$ with $h$ pendant
vertices and degree sequence 
\begin{equation*}
\pi =(\underset{n-h}{\underbrace{d_{1},d_{2},..,d_{n-h-1},d_{n-h}}},\underset%
{h}{\underbrace{1,...,1}}),\quad n\geq 4,n-h\geq 2,h\geq 1
\end{equation*}%
and let us consider graphs $G\in C_{\pi }$ with maximum vertex degree upper
bounded by $d_{n-h}+d_{n-h-1}$, i.e. $d_{1}<$ $d_{n-h}+d_{n-h-1}$, or
equivalently 
\begin{equation}
1+d_{1}\leq d_{n-h}+d_{n-h-1}.  \label{eq:limitazione}
\end{equation}%
\noindent For $G\in C_{\pi }$, we note that this constraint is always
satisfied, for example, if the maximum vertex degree is at most three, as
for some graphs of chemical interest where the maximum degree is four.

\noindent We observe that for $i,j=1,...,n-h$ and $(v_{i},v_{j})\in E:$%
\begin{equation*}
d_{n-h}+d_{n-h-1}\leq d_{i}+d_{j}\leq d_{1}+d_{2},
\end{equation*}%
\noindent while for $i=n-h+1,...,n;$ $\ \ \ j=1,...,n-h$ $\ $and $%
(v_{i},v_{j})\in E:$%
\begin{equation*}
1+d_{n-h}\leq d_{i}+d_{j}\leq 1+d_{1}.
\end{equation*}

\noindent Furthermore, inequality (\ref{eq:limitazione}) assures that the
above intervals are concatenated so that the vector $\mathbf{x}\in 
\mathbb{R}
^{m}$ \ can be arranged in nonincreasing order with the $h$ pendant vertices
in the last $h$ positions.

\noindent Setting $m_{1}=d_{n-h}+d_{n-h-1}$, $\ \ m_{2}=1+d_{n-h}$, $\ \
M_{1}=d_{1}+d_{2}$, \ $M_{2}=1+d_{1},$ let us consider the set%
\begin{equation*}
S_{a}^{m-h}=%
\begin{array}{c}
\Sigma _{a}\cap \left\{ \mathbf{x}\in \mathbb{R}^{n}:M_{1}\geq x_{1}\geq
...x_{m-h}\geq m_{1},\right. \\ 
\left. M_{2}\geq x_{m-h+1}\geq ...x_{m}\geq m_{2}\right\}%
\end{array}%
.
\end{equation*}

\noindent Applying Corollaries \ref{cor:dueintervalli} and \ref%
{dueintervalliminimo} we can compute maximal and minimal elements of $%
S_{a}^{m-h}$ with respect to the majorization order and from (\ref{SG}) we
obtain: 
\begin{equation}
\frac{\left\Vert x_{\ast }(S_{a}^{m-h})\right\Vert _{2}^{2}-\overset{n}{%
\underset{i=1}{\sum }}d_{i}^{3}}{2}\leq S(G)\leq \frac{\left\Vert x^{\ast
}(S_{a}^{m-h})\right\Vert _{2}^{2}-\overset{n}{\underset{i=1}{\sum }}%
d_{i}^{3}}{2},  \label{bounds}
\end{equation}%
where $\left\Vert \cdot \right\Vert _{2}$ stands for the euclidean norm.

\noindent In spite of inequalities (\ref{confronto}) and (\ref{confronto1}),
these bounds can't be worse than those in \cite{Grassi}, and they are often
sharper.

It is noteworthy that both equalities in (\ref{bounds}) are attained if and
only if the set $S_{a}^{\left[ h\right] }$ reduces to a singleton, that is,
by Remark \ref{rem:m=M}$,$ $m_{i}=M_{i},i=1,2.$

The condition $m_{2}=1+d_{n-h}=M_{2}=1+d_{1}$ implies that in $G(V,E)$ all
non-pendant vertices have the same degree. Some examples of this kind of
graphs are:

$i)$ all trees with degree sequence 
\begin{equation}
\pi =\left( \underset{r}{\underbrace{k,...,k}},\underset{rk-2r+2}{%
\underbrace{1,...,1}}\right) ,  \label{PC}
\end{equation}%
including, as particular case, for $k=2$, the path.

$ii)$ graphs obtained by adding the same number $s$ of pendant vertices to
each vertex of a $k-$regular graph on $r$ vertices, being $kr$ even, $2\leq
k\leq r-1,$ i.e. 
\begin{equation*}
\pi =\left( \underset{r}{\underbrace{k+s,...,k+s}},\underset{sr}{\underbrace{%
1,...,1}}\right) .
\end{equation*}

Computing $S(G)$, from Remark \ref{rem:m=M} and (\ref{bounds}), we get $%
k\left( 2kr-2r-k+2\right) $ and $\frac{1}{2}r\left( 2s+ks+k^{2}\right)
\left( k+s\right) $ respectively.


\noindent In the following we provide some significant examples, computing
bounds for graphs belonging to $C_{\pi }$ and satisfying (\ref%
{eq:limitazione}). Furthermore, a comparison with some other known bounds
(see \cite{Bollobas}, \cite{Das-Gut-09}, \cite{Lu}, \cite{Yan} and \cite%
{Zhao}) are provided.

\vskip0.5cm \noindent \textbf{Example 1.} Let us consider the classes of
trees $T_{t,s}$ with degree sequences $\pi _{i}$ $(i=1,2,3)$ given by:

\begin{description}
\item[i)] $\pi _{1}=\left( t,\underset{t}{\underbrace{s,....,s}},\underset{%
t\left( s-1\right) }{\underbrace{1,....,1}}\right), \,\, 2\leq s<t<2s $

\item[ii)] $\pi _{2}=\left( \underset{t}{\underbrace{s,....,s}},t\underset{%
t\left( s-1\right) }{,\underbrace{1,....,1}}\right), \,\, s >t \geq 2 $

\item[iii)] $\pi _{3}=\left( \underset{t+1}{\underbrace{t,....,t}},\underset{%
t(t-1)}{\underbrace{1,...,1}}\right), \,\, t \geq 2 $
\end{description}

\noindent Case $i)$.

\begin{center}
\begin{tabular}{|c|c|}
\hline
$M_1=t+s$ & $m_1=2s$ \\ \hline
$M_2=t+1$ & $m_2=s+1$ \\ \hline
$m=ts$ & $h=t(s-1)$ \\ \hline
\end{tabular}
\end{center}

\noindent Applying Corollary \ref{cor:dueintervalli} and Remark \ref%
{rem:a=a*} it follows that: 
\begin{equation*}
x^{\ast }\left( T_{t,s}\right) =\left[ \underset{t}{\underbrace{\left(
t+s\right) ,.....,\left( t+s\right) }},\underset{st-t}{\underbrace{\left(
s+1\right) ,....,\left( s+1\right) }}\right]
\end{equation*}%
\noindent while from (\ref{m_1 le M_2}), (\ref{M_2<m_1}) and Remark \ref%
{re:componenti intere} we get 
\begin{equation}
x_{\ast }\left( T_{t,s}\right) =\left\{ 
\begin{array}{cc}
\left[ \underset{t}{\underbrace{2s,....,2s}},\underset{t(t-s)}{\underbrace{%
s+2,.....,s+2},}\underset{t(2s-t-1)}{\underbrace{s+1,.....,s+1}}\right] & 
\text{ if }t<2s-1 \\ 
\left[ \underset{t}{\underbrace{2s,....,2s}},\underset{st-t}{\underbrace{%
s+2,.....,s+2}}\right] & \text{ if }t=2s-1%
\end{array}%
\right. .  \label{eq:minimo1}
\end{equation}

\noindent Taking into account (\ref{bounds}), the following inequalities
hold:%
\begin{equation}
\left\{ 
\begin{array}{cc}
\frac{1}{2}t\left( 3t-t^{2}-5s+2st+3s^{2}\right) \leq S(T_{t,s})\leq
ts\left( s+t-1\right) & \text{ if }t<2s-1 \\ 
\frac{1}{2}\left( 2s-1\right) \left( 3s+3s^{2}-4\right) \leq S(T_{t,s})\leq
s\left( 2s-1\right) \left( 3s-2\right) & \text{ if }t=2s-1%
\end{array}%
\right. .  \label{stima1}
\end{equation}

\noindent We note that in (\ref{eq:minimo1}) the right-hand equality holds
if $T_{t,s}$ is the tree obtained by the union of $t$ stars, each one of
order (see Figure 1).

\begin{figure}[!htb]
\centering \includegraphics[height=5cm]{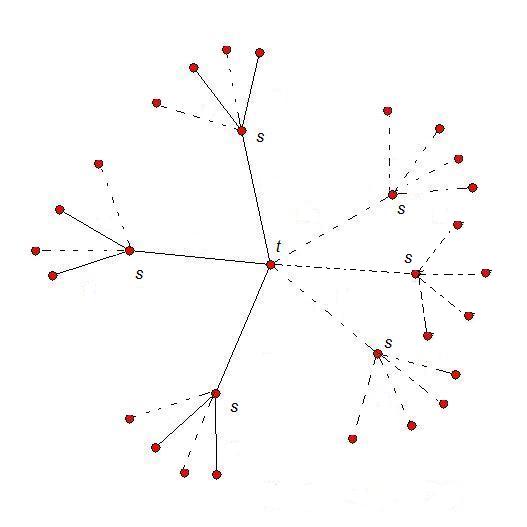}
\caption{Example illustrating tree $T_{t,s}$ for $2\leq s<t<2s $. }
\label{fig:Star1}
\end{figure}

\noindent Case $ii)$.

\begin{center}
\begin{tabular}{|c|c|}
\hline
$M_1=2s$ & $m_1=t+s$ \\ \hline
$M_2=s+1$ & $m_2=t+1$ \\ \hline
$m=ts$ & $h=t(s-1)$ \\ \hline
\end{tabular}
\end{center}

By Corollary \ref{cor:dueintervalli} follows 
\begin{equation*}
x^{\ast }\left( T_{t,s}\right) =\left[ \underset{t}{\underbrace{2s,....,2s}},%
\underset{st-2t}{\underbrace{\left( s+1\right) ,....,\left( s+1\right) }},%
\underset{t}{\underbrace{\left( t+1\right) ,....,\left( t+1\right) }}\right]
\end{equation*}

\noindent while from Remark (\ref{remark:a=atilde}) we get 
\begin{equation*}
x_{\ast }\left( T_{t,s}\right) =\left[ \underset{t}{\underbrace{s+t,....,s+t}%
},\underset{st-t}{\underbrace{(s+1),....,(s+1)}}\right] .
\end{equation*}

\noindent Taking into account (\ref{bounds}), the following inequalities
hold:%
\begin{equation}
ts(s+t-1)\leq S(T_{t,s})\leq t(t-2s+2s^{2})  \label{stima2}
\end{equation}

\noindent We note that the left-hand equality holds if $T_{t,s}$ is the tree
obtained by the union of $t$ stars each one of order $s$ (see Figure 2). 
\begin{figure}[tbh]
\centering \includegraphics[height=7cm]{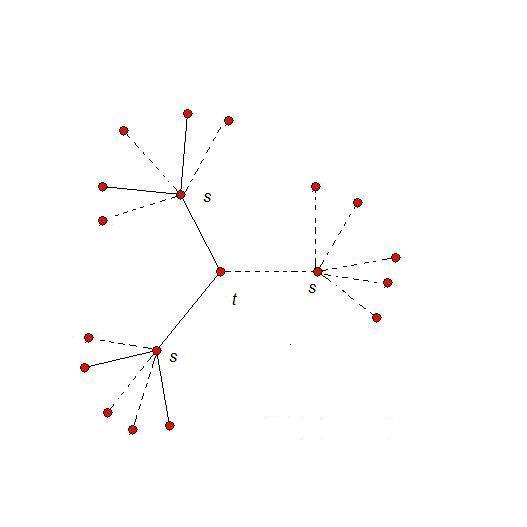}
\caption{Example illustrating tree $T_{t,s}$ for $s>t\geq 2$. }
\label{fig:Star2}
\end{figure}

\noindent Case $iii)$. This is a particular case of (\ref{PC}), for $k=t$
and $r=t+1$, such that 
\begin{equation}
S(T_{t,s})=2t^{3}-t^{2}.  \label{stima3}
\end{equation}

\noindent Finally we observe that for the class of trees with degree
sequence $\pi _{1},\pi _{2}$ or $\pi _{3}$, our upper bounds always perform
better than those in \cite{Das-Gut-09}. Indeed, in the presence of pendant
vertices and with $m=ts$ and $n=ts+1,$ the bound in \cite{Das-Gut-09}
becomes: 
\begin{equation}
S(G)\leq 2m^{2}-(n-1)m=t^{2}s^{2}  \label{eq:SGDasGut}
\end{equation}%
which is always greater than the upper bound in (\ref{stima1}), (\ref{stima2}%
), (\ref{stima3}). \newpage

\vskip0.5cm \noindent \textbf{Example 2}. Let us consider a unicyclic graph $%
G$, i.e. a graph with $n=m$ having the following degree sequence 
\begin{equation*}
\pi =(3,3,3,3,2,2,2,2,2,1,1,1,1).
\end{equation*}%
Being (\ref{eq:limitazione}) satisfied, by Remark \ref{re:componenti intere}%
, (\ref{bounds}) gives%
\begin{equation*}
64\leq S(G)\leq 74.
\end{equation*}

\noindent The comparison (see Table 1) with bounds in \cite{Bollobas}, \cite%
{Das-Gut-09} , \cite{Grassi}, \cite{Lu} and \cite{Yan} shows that our bounds
always perform better. Indeed we obtain:

\begin{table}[!ht]
\centering
\begin{tabular}{|l|l|l|}
\hline
Bounds & Lower & Upper \\ \hline
ours & 64 & 74 \\ \hline
\cite{Bollobas} & x & 277.9 \\ \hline
\cite{Das-Gut-09} & x & 182 \\ \hline
\cite{Grassi} & 61.462 & 77 \\ \hline
\cite{Lu} & -28 & 76 \\ \hline
\cite{Yan} & 64 & 92 \\ \hline
\end{tabular}%
\caption{Lower and upper bounds for $S(G)$ }
\end{table}

\vskip0.5cm \noindent \textbf{Example 3}. Consider the graphs $G$ and $H$
with degree sequences $\pi _{1}=(3,2,2,1)$ and $\pi _{2}=(3,3,3,3,2,1,1)$
respectively, as in Examples 2.2 and 2.3 in \cite{Grassi}. Besides the
bounds discussed in \cite{Grassi}, we add the comparison with those in \cite%
{Das-Gut-09}, \cite{Yan} and \cite{Zhao}. Observing that $G$ is a unicyclic
graph $(m=n)$ and $H$ is a bicyclic graph $(m=n+1)$, both with pendant
vertices, bounds in \cite{Yan} and \cite{Zhao} can also be respectively
properly applied. \noindent Computing bounds for $S(G)$, we have:

\begin{table}[!th]
\centering
\begin{tabular}{|l|l|l|}
\hline
R$\text{ef.}$ & L$\text{ower}$ & U$\text{pper}$ \\ \hline
$\text{ours}$ & $19$ & $20$ \\ \hline
\cite{Bollobas} & $\text{x}$ & $22.511$ \\ \hline
\cite{Das-Gut-09} & $\text{x}$ & $20$ \\ \hline
\cite{Grassi} & $18.5$ & $20$ \\ \hline
\cite{Lu} & $18$ & $22$ \\ \hline
\cite{Yan} & $19$ & $19$ \\ \hline
\end{tabular}%
\caption{Lower and upper bounds for $\ S(G)$. }
\end{table}

\newpage

\noindent Our bounds are sharper than \cite{Bollobas}, \cite{Grassi} and 
\cite{Lu}. The best one is provided by \cite{Yan} and has been specifically
constructed for this class of graph.

\noindent Computing bounds for $S(H)$, we have:

\begin{table}[th]
\centering
\begin{tabular}{|l|l|l|}
\hline
$\text{ref.}$ & $\text{lower}$ & $\text{upper}$ \\ \hline
our & $54$ & $58$ \\ \hline
\cite{Bollobas} & $\text{x}$ & $99.75$ \\ \hline
\cite{Das-Gut-09} & $\text{x}$ & $80$ \\ \hline
\cite{Grassi} & $51.25$ & $58$ \\ \hline
\cite{Lu} & $40$ & $59$ \\ \hline
\cite{Zhao} & $50$ & $68$ \\ \hline
\end{tabular}%
\caption{Lower and upper bounds for $S(H)$. }
\end{table}

\noindent Note that our bounds perform better than all the others and in
particular better than \cite{Zhao} which is properly designed for bicyclic
graphs as $H$ is.

\section{Conclusion}

The purpose of this paper is to establish maximal and minimal vectors with
respect to the majorization order under sharper constraints than those
presented by Marshall and Olkin in \cite{Marshall}. We have shown how these
results can provide a simple methodology for localizing the second Zagreb
index of a particular class of graphs. Some numerical examples have been
discussed, showing that our bounds often provide sharper bounds than those
in the literature. Moreover, in network analysis, there are a variety of
potential applications for this kind of approach, considering other
topological indices which can be defined by a suitable Schur-convex
function. \newpage 

\end{document}